\documentclass[review]{elsarticle1}

\usepackage{etex}

\usepackage{psfrag}
\usepackage{cancel}
\usepackage{latexsym}
\usepackage{graphicx}
\usepackage{float}
\usepackage{colortbl}
\usepackage{fancyhdr}
\usepackage{subcaption}
\usepackage{hyperref}
\usepackage{enumerate}
\usepackage{multicol}
\usepackage{float}
\usepackage{mathtools, cuted}
\usepackage{graphicx,color}
\usepackage{amsmath,amssymb,amsfonts,booktabs, mathrsfs}
\usepackage{bbm,dsfont}
\usepackage[normalem]{ulem}
\usepackage{setspace}

\usepackage[margin=2.5cm]{geometry}
\usepackage[usenames,dvipsnames,svgnames,table]{xcolor}
\usepackage{tikz}
\usepackage{mhchem}
\usetikzlibrary{calc,arrows,automata,positioning}
\usepackage{pgfplots}
\usepgfplotslibrary{patchplots}
\usepackage{todonotes}

\newcommand{\M}{\mathbb{M}}

\newtheorem{theorem}{Theorem}

\newtheorem{lemma}[theorem]{Lemma}
\newtheorem{proposition}[theorem]{Proposition}

\newtheorem{remark}{Remark}
\newtheorem{problem}{Problem}

\newcommand{\ds}{\displaystyle}

\newcommand{\enma}[1]   {\ensuremath{#1}}

\newcommand{\non}{\nonumber}

\newcommand{\beq}{\begin{equation}}
\newcommand{\eeq}{\end{equation}}
\newcommand{\bseq}{\begin{subequations}}
\newcommand{\eseq}{\end{subequations}}
\newcommand{\beqn}{\begin{eqnarray}}
\newcommand{\eeqn}{\end{eqnarray}}
\newcommand{\ba}{\begin{array}}
\newcommand{\ea}{\end{array}}
\newcommand{\bct}{\begin{center}}
\newcommand{\ect}{\end{center}}
\newcommand{\btmz}{\begin{itemize}}
\newcommand{\etmz}{\end{itemize}}
\newcommand{\benum}{\begin{enumerate}}
\newcommand{\eenum}{\end{enumerate}}

\newcommand{\mc}{\mathcal}

\newcommand{\R}{{\mathbb R}}
\newcommand{\C}{{\mathbb C}}

\newcommand{\D}{{\mathbb D}}

\newcommand{\cH}{\enma{\mathcal H}}
\newcommand{\cL}{\enma{\mathcal L}}

\newcommand{\norm}[1]{\| #1 \|}                 

\newcommand{\diag}      {\enma{\mathrm{diag}}}

\newcommand{\inner}[2]{\left\langle #1,#2 \right\rangle}

\newcommand{\matbegin}{
        \left[
}
\newcommand{\matend}{
        \right]
}

\newcommand{\tbo}[2]{
  \matbegin \begin{array}{c}
       #1 \\ #2
       \end{array} \matend }
\newcommand{\thbo}[3]{
  \matbegin \begin{array}{c}
       #1 \\ #2 \\ #3
       \end{array} \matend }

\newcommand{\thbth}[9]{
 \matbegin \begin{array}{ccc}
                #1 & #2 & #3 \\
                #4 & #5 & #6 \\
                #7 & #8 & #9
                \end{array}\matend}

\newcommand{\be}{\begin{equation}}
\newcommand{\ee}{\end{equation}}

\newcommand{\cplxs}{ C\kern -.35em \rule{0.03 em}{.7 ex}~   }

\def\complex{\hbox{C\kern -.45em \rule{0.03 em}{1.5 ex}}~}

\newcommand{\bi}{\begin{itemize}}
\newcommand{\ei}{\end{itemize}}

\newtheorem{assumption}{Assumption}

\DeclareMathOperator*{\argmin}{argmin}
\DeclareMathOperator*{\minimize}{minimize}
\DeclareMathOperator*{\maximize}{maximize}

\newcommand{\mre}{\mathrm{e}}
\newcommand{\mrj}{\mathrm{j}}

\newcommand{\bbR}{\mathbb{R}}
\newcommand{\DefinedAs}[0]{\mathrel{\mathop:}=}

\newcommand{\vsp}{\vspace{0.1cm}}

\graphicspath{{Figures/} }

\newproof{proof}{Proof}

\begin{document}

\begin{frontmatter}

\title{\LARGE  On the optimal control problem for a class \\of monotone bilinear systems}
\author[umn]{Neil K.\ Dhingra\corref{cor1}\tnoteref{fumn}}
\author[eth]{Marcello Colombino\tnoteref{feth}}
\author[umn]{Mihailo R.\ Jovanovi\'{c}\tnoteref{fumn}}
\author[lund]{\\[0.1cm]Anders Rantzer}
\author[eth]{and Roy S.\ Smith\tnoteref{feth}}

\tnotetext[fumn]{Supported in part by the National Science Foundation under Awards ECCS-1407958 and CNS-1544887.}
\tnotetext[feth]{Supported in part by the Swiss National Science Foundation under Award
2-773337-12.}

\address[umn]{Department of Electrical and Computer Engineering, University of Minnesota, Minneapolis, MN 55455}
\address[eth]{Automatic Control Laboratory, Swiss Federal Institute of Technology, Z\"urich, Switzerland}
\address[lund]{Automatic Control LTH, Lund University, SE-221 00 Lund, Sweden}

\cortext[cor1]{Corresponding author}

\begin{abstract}
We consider a class of monotone systems in which the control signal multiplies the state. Among other applications, such bilinear systems can be used to model the evolutionary dynamics of HIV in the presence of combination drug therapy. For this class of systems, we formulate an infinite horizon optimal control problem, prove that the optimal control signal is constant over time, and show that it can be computed by solving a finite-dimensional non-smooth convex optimization problem. We provide an explicit expression for the subdifferential set of the objective function and use a subgradient algorithm to design the optimal controller. We further extend our results to characterize the optimal robust controller for systems with uncertain dynamics and show that computing the robust controller is no harder than computing the nominal controller. We illustrate our results with an example motivated by combination drug therapy.
\end{abstract}

\end{frontmatter}

	\vspace*{-2ex}
\section{Introduction}

This work is motivated by recent developments on the modeling and design of combination drug therapy design for HIV treatment~\cite{hernandez2011discrete,hernandez2013optimal,colaneri2014convexity,ranber14,jonranmur13,jonmatmur13,jonmatmur14,dhijovACC15,dhingra2016convexity}. These advancements provide a framework for modeling the evolutionary dynamics of the population of different HIV mutants within an infected patient as a bilinear monotone system in which drug therapy affects the death rate of specific mutagens. The drug cocktail prescribed to the patient over the course of their treatment can be interpreted as a time-varying decentralized controller for the positive system that describes the virus mutation.

A system is called positive if, for every nonnegative initial condition and nonnegative input signal, its state and output remain nonnegative for all time~\cite{farina2011positive}. Positive systems and their nonlinear counterpart, monotone systems, arise naturally when modeling physical and biological systems with inherently nonnegative states, e.g., the viral population of HIV. Other common examples include Markov chains and models of chemical reaction networks, transportation networks, population dynamics, and heat transfer; in these, the states represent probabilities, concentrations, traffic density, species counts, and temperature, respectively.


Positive systems are not only interesting from a practical point of view but they also have rich structure and subtle system-theoretic properties. The study of positive systems dates back to the early 1900{s}, when Perron and Frobenius explored the spectral properties of nonnegative matrices. In recent years positive systems have gained renewed attention because several problems that are intractable in general greatly simplify for this class of systems.

It has been recently shown that the design of unconstrained decentralized controllers for positive systems is convex~\cite{rantzer2015scalable,rantzer2016kalman}. However, many problems, such as the design of drug therapy for HIV, impose additional structural constraints which cannot be handled by the LMI approach proposed in~\cite{rantzer2016kalman}. To overcome this challenge, a design of $\cL_1$ and $\cH_\infty$ controllers which satisfy such structural constraints but achieve suboptimal performance has been examined in~\cite{jonranmur13,jonmatmur13,jonmatmur14}. More recently, building on~\cite{colaneri2014convexity,ranber14}, convexity of the structured decentralized $\cH_2$ and $\cH_\infty$ optimal control problems for positive systems was established in~\cite{dhingra2016convexity}.


While references~\cite{jonranmur13,jonmatmur13,jonmatmur14,dhijovACC15,dhicoljovCDC16,dhingra2016convexity} either assume a constant control signal or use heuristics to introduce time dependence, we show that such a constant input is in fact optimal  for the induced power norm. We cast the optimal synthesis problem of constant control inputs as a finite-dimensional non-smooth convex optimization problem and develop an algorithm for designing the optimal controller. Finally, we exploit the advancements in~\cite{ColSmi:2014:IFA_4769,ColSmi:2016:IFA_5242} to design robust controllers which guard against model uncertainty.

Our presentation is organized as follows. In Section~\ref{sec.prel}, we formulate the optimal control problem for the class of bilinear positive systems that we study. In Section~\ref{sec.main}, we prove that a constant control input solves our optimal control problem and develop a subgradient algorithm. In Section~\ref{sec.robust}, we show that designing the robust controller for systems with model uncertainty is no harder than computing a nominal controller. In Section~\ref{sec.hiv}, we apply our results to an example inspired by combination drug therapy design for HIV. Finally, we conclude the paper in Section~\ref{sec.conc}.

	\vspace*{-2ex}
\section{Preliminaries and problem formulation}
	\label{sec.prel}

In this section, we provide necessary background material, introduce the class of bilinear positive systems that we study, and formulate the optimal control problem.

	\vspace*{-2ex}
\subsection{Preliminaries}

The set of real numbers is denoted by $\R$ and $\R_+$ $(\R_{++})$ is the set of nonnegative (positive) reals. The set of $n\times  n$ Metzler matrices (matrices with nonnegative off-diagonal elements) is denoted by $\M^n$. The set of $n\times n$ nonnegative (positive) diagonal matrices is denoted by $\D^n_{+}$ $(\D^n_{++})$. We use $\bar{\sigma}(A)$ to indicate the largest singular value of $A$.

The space of square integrable signals is denoted by $\mc L_2$. The inner product in this space is given by
$$
\inner{u}{v}_2 \;\DefinedAs\; \int_0^\infty  u^T(t) \, v(t) \,\mathrm{d} t
$$
with the associated norm $\norm{v}_{2}^2 = \inner{v}{v}_2$. The power semi-norm of a signal $v$ is
\be\label{eq.power_seminorm}
\norm{v}_{\textnormal{pow}}^2
\;\DefinedAs\;
\limsup_{T \to \infty}
~
\dfrac{1}{T} \int_0^T  v^T(t) \, v(t) \,\mathrm{d} t.
\ee
The space of trigonometric polynomials is defined as
\be
	\mc {T}
	\;\DefinedAs\;
	\left\{ g: \R\to\R^n\, \left |~  g(t) = \sum_{k \, = \, 1}^N \alpha_k \mathrm e ^{\mrj\lambda_k t}, \right.
  \lambda_k \in \R,\, \alpha_k\in\C \right\},
  \non
\ee
where $\mrj$ is the imaginary unit. The closure of $\mc{T}$ with respect to the metric $\norm{f-g}_\textnormal{pow}$ is given by the space of Besicovitch almost periodic functions $\mc B_2$~\cite{besioovitoh1954almost}. This is a Hilbert space with the inner product~\cite{corduneanu2009almost,constanda2010integral}
$$
\inner{u}{v} \; \DefinedAs \;
\limsup_{T \to \infty}  ~ \dfrac{1}{T} \int_0^T  u^T(t) \, v(t) \,\mathrm{d} t
$$
and the norm $\|\cdot\|_\textnormal{pow}$. The mean of a signal $v \in \mc B_2$
$$
M(v) \; \DefinedAs \; \lim_{T \, \to \, \infty}  \; \dfrac{1}{T} \int_0^T  v(t) \,\mathrm{d} t
$$
is well defined for every $v \in \mc B_2$. Furthermore, each $v \in \mc B_2$ can be decomposed uniquely as $v = \bar v + \tilde v$, where $\bar v$ is a constant signal given by $\bar{v} = M(v)$ and $M(\tilde v) =0$. Note that the inner product between a constant signal $\bar v$ and a zero-mean signal $\tilde v$ is zero, i.e., $\inner{\bar v}{\tilde v} = 0$.

The space $\mc B_2$ contains all bounded $\mc L_2$ signals, periodic signals, and almost periodic signals. At the same time, it alleviates challenges arising from the fact that the space of signals with bounded power norm is not a Hilbert space; for additional discussion see~\cite{mari1996counterexample}.

	\vspace*{-2ex}
\subsection{Problem formulation}
Consider the bilinear system
\begin{subequations}
	\label{eq.sys-z}
\be
\dot x
\; = \;
\left(
A
\,+
\,D(u)
\right)
x
\,+ \,
B \, d
\label{eq.sys}
\ee
where $A\in\R^{n\times n}$, $B\in\R^{n\times q}$, and $D$: $\R^m\to{\D^n}$ is a linear operator. {For given control and disturbance signals $u \in \mc{B}_2$ and $d \in \mc{B}_2$, we associate the performance output,}
\be
z_{u,d} \; = \;
\tbo{Q^{1\over 2}}{0}x \; + \; \tbo{0}{R^{1\over 2}}u
\label{eq.z}
\ee
\end{subequations}
{with~\eqref{eq.sys}, where $Q \succeq 0$ and $R \succ 0$ are the state and control weights.}

\begin{assumption}\label{ass.pos}
{$A$ is a Metzler matrix, $D$: $\R^m\to \D^n$ is a linear operator, $B$ and $Q$ are matrices with nonnegative entries, and there is a constant stabilizing control input $\bar u$ for~\eqref{eq.sys}.}
\end{assumption}

Under Assumption~\ref{ass.pos},~\eqref{eq.sys-z} is a positive system. This implies that for every control input $u$, every nonnegative disturbance $d$, and every nonnegative initial condition $x(0)$, the state $x$ and the output $z_{u,d}$ of system~\eqref{eq.sys-z} remain nonnegative at all times.

The induced power norm of a stable system~\eqref{eq.sys-z} is,
\be
J(u)
\;\DefinedAs\;
\sup_{\norm{d}_{\textnormal{pow}}^2 \, \leq \, 1}\;
\norm{z_{u,d} }^2_{\textnormal{pow}}
\label{eq.J}
\ee
and it quantifies the response to the worst case persistent disturbance $d$. For unstable open-loop systems~\eqref{eq.sys-z}, there may be no stabilizing control input $u$ in $\mc{L}_2$. Thus, the $\mc{L}_2$-induced gain does not provide a suitable measure of input-output amplification for~\eqref{eq.sys-z} and $J(u)$ represents an appropriate generalization of the $\cH_\infty$ norm for this class of bilinear positive systems.
	
We now formulate the optimal control problem.

\begin{problem}\label{prob.opt}
{Design a
stabilizing bounded control signal $u \in \mc{B}_2$ to minimize $J(u)$ for bilinear positive system~\eqref{eq.sys-z}}.
 \end{problem}

	\vspace*{-2ex}
\section{Solution to the optimal control problem}
	\label{sec.main}

{In this section, we prove that a constant control input solves Problem~\ref{prob.opt} and provide a subgradient algorithm for finding the optimal solution.}

	\vspace*{-2ex}
\subsection{Main result}

Since $J(u)$ is given by~\eqref{eq.J}, any $u^\star$ which solves Problem~\ref{prob.opt} satisfies $\norm{z_{u^\star, d}}_{\textnormal{pow}}^2 \leq J(u^\star) \leq J(u)$ for all $u \in \mc{B}_2$ and $d \in \mc{B}_2$. In particular,
$
	\norm{z_{u^\star, d}}_{\textnormal{pow}}^2 \leq J(\bar u)
$
where $\bar u$ is a constant control input. As shown in~\cite{tanaka2011bounded}, for constant control inputs the worst-case disturbance $\bar{d}$ is also constant, i.e., $J(\bar u) = \norm{z_{\bar{u}, \bar d}}_{\textnormal{pow}}^2$. In what follows, we show that $\norm{z_{u, \bar d}}_{\textnormal{pow}}^2$ is a convex function of $u$, that a constant $\bar u$ minimizes it, and, thus, that a constant control input solves Problem~\ref{prob.opt}.

We first establish convexity of $\norm{z_{u, \bar d}}^2_{\textnormal{pow}}$.

\begin{lemma}\label{lem.convex.l2}
Let $d(t) = \bar d$ be a constant non-negative disturbance. Then, under Assumption~\ref{ass.pos}, the power norm of the output $\norm{z_{u, \bar d} }_{\textnormal{pow}}^2$ is a convex function of $u\in\mc{B}_2$.
\end{lemma}

\begin{proof}
By~\cite[Theorem 2]{colaneri2014convexity} every component of $x(t)$ is a convex function of the control input $u\in\mc L_2[0,t]$. The power norm of the performance output~\eqref{eq.z} is given by,
\[
	\norm{z_{u, \bar d} }_{\textnormal{pow}}^2
	\,=\,
	\limsup_{T \to \infty}
	~
	\dfrac{1}{T} \int_0^T
	\left(
	x^T(t)Qx(t)
	\,+\,
	u^T(t)Ru(t)
	\right)
	\mathrm{d} t.
\]
{The matrix $R$ is positive definite so the second term on the right-hand side is a convex function of $u$. Furthermore, since the matrix $Q \succeq 0$ has nonnegative entries, $v^TQv$ is nondecreasing in the elements of $v$. Thus, the composition rules for convex functions~\cite{boyvan04} imply that $x^T(t)Qx(t)$ is a convex function of $u\in\mc L_2[0,t]$. }

{Let $P_T$ denote the mapping that truncates the support of a signal $v$ to $[0,T]$. If there is $T$ such that $P_Tv$ is not an $\cL_2$ signal, the power seminorm~\eqref{eq.power_seminorm} of $v$ is infinite. This implies that, for any $v \in \mc{B}_2$, $P_Tv \in \mc{L}_2[0,T]$. Thus, any $v \in \mc{B}_2$ can be written as $\lim_{T \to \infty} P_Tv$ of $P_Tv \in \mc{L}_2[0,T]$.} Since both the integral and the limit preserve convexity, $\norm{z_{u, \bar d} }_{\textnormal{pow}}^2$ is a convex function of $u \in \mc{B}_2$.
\end{proof}

\begin{lemma}\label{lem}
Let $\bar u$ be a stabilizing constant control input for~\eqref{eq.sys} and $\bar d$ be a constant non-negative disturbance. Then, the directional derivative of $\norm{z_{u, \bar d} }_{\textnormal{pow}}^2$ evaluated at $\bar u$  is zero for any bounded zero-mean variation $\tilde u \in \mc{B}_2$.
\end{lemma}

\begin{proof}
The dynamics~\eqref{eq.sys} with control input $\bar u + \varepsilon \tilde{u}$ and constant disturbance $\bar d$ are
\be
\dot x
	\; = \;
\left(
A
\,+
\,D(\bar u \, + \, \varepsilon \tilde u)
\right)
x
\,+ \,
B \, \bar d.
\label{eq.sysall}
\ee
Since the unperturbed system (with $\varepsilon=0$) is exponentially stable and the solution $x(t)$ is continuous in $\varepsilon$,~\eqref{eq.sysall} represents a system in a regularly perturbed form~\cite{kokkharei99,kha96}. The Taylor series expansion can be used to write the solution to the perturbed dynamics as,
\be\label{eq.approx}
    x(t)
    \;=\;
    \bar x(t)
    \;+\;
    \varepsilon \, \tilde x(t)
    \;+\;
    \mathcal{O}(\varepsilon^2)
\ee
where $\bar x$ is the nominal solution that solves~\eqref{eq.sysall} for $\varepsilon = 0$
	\begin{subequations}
	\label{eq.x0x1}
	\be
\dot{\bar{x}}
 \; = \;
\left(
A
\,+\,
D(\bar u)
\right)
\bar{x}
\,+\,
B \, \bar d
\label{eq.x0}
	\ee
and $\tilde x$ represents the first-order correction. By~\cite[Theorem 10.2]{kha96}, $\tilde x$ is determined by the solution to a differential equation corresponding to the $\mathcal{O}(\varepsilon)$ terms in the expression obtained by substituting~\eqref{eq.approx} into equation~\eqref{eq.sysall},
	\be
\dot{\tilde{x}}
\; = \;
\left(
A
\,+
\,D(\bar u )
\right)
\tilde{x}
\,+ \,
D(\tilde u) \, \bar{x}.
\label{eq.x1}
\ee
\end{subequations}
Both~\eqref{eq.x0} and~\eqref{eq.x1} are stable LTI systems. Thus, for the constant disturbance $d(t)=\bar d$, the solution $\bar{x}$ to~\eqref{eq.x0} is asymptotically constant. Since $\tilde u$ is zero-mean, the signal $D(\tilde u) \, \bar{x}$ is also zero-mean. The dynamics~\eqref{eq.x1} are a stable LTI system forced by a zero-mean input and therefore $\tilde{x}$ is also zero-mean.


Because $\bar x$ is asymptotically constant and $\tilde x$ is zero-mean, $\inner{Q^{1/2}\bar{x}}{Q^{1/2}\tilde{x}} = 0$. Similarly, $\inner{R^{1/2}\bar u}{R^{1/2}\tilde u}= 0$. Furthermore, $\norm{z_{u, \bar d}}_{\textnormal{pow}}^2 = \inner{Q^{1/2}{x}}{Q^{1/2}{x}} + \inner{R^{1/2}u}{R^{1/2}u}$, and we have,
\be
\label{eq.oe2}
\ba{rcl}
\norm{z_{(\bar u + \varepsilon\tilde u), \bar d}}_{\textnormal{pow}}^2
\; - \;
\norm{z_{\bar u, \bar d}}_{\textnormal{pow}}^2
&\!\!=\!\!&
2 \, \varepsilon
\left(
\inner{Q^{1/2}\bar{x}}{Q^{1/2}\tilde{x}}
\;+\;
\inner{R^{1/2}\bar u}{R^{1/2}\tilde u}
\right)
\;+\;
\mathcal{O}(\varepsilon^2)
\\[0.1cm]
&\!\!=\!\!&
\mathcal{O}(\varepsilon^2).
\ea
\ee
Thus, the first order correction to $\norm{z_{u, \bar d}}_{\textnormal{pow}}^2$ evaluated at a constant control input $\bar{u}$ is zero, which completes the proof.
\end{proof}

	\begin{remark}
Lemma~\ref{lem} does not need Assumption~\ref{ass.pos} and it holds for all bilinear systems of the form~\eqref{eq.sys} for which there is a stabilizing constant control input.
	\end{remark}
	
Lemma~\ref{lem} implies that a constant control signal $u^\star$
$$
	u^\star
	\;\in\;
	\argmin_{\bar u \textnormal{ constant} }
	~
	\norm{z_{\bar u, \bar d} }_{\textnormal{pow}}^2,
$$
provides a local minimum for $\norm{z_{u, \bar d} }_{\textnormal{pow}}^2$. Based on Lemma~\ref{lem.convex.l2}, this quantity is convex with respect to $u$. Thus, $u^\star$ is a global minimizer. The following theorem relates the power norm of the output of system~\eqref{eq.sys-z} subject to a constant disturbance, $\norm{z_{u, \bar d} }_{\textnormal{pow}}^2$, with the worst case power norm amplification $J(u)$ and shows that the constant control signal $u^\star$ solves Problem~\ref{prob.opt}.

\begin{theorem}\label{thm.const.opt}
Let Assumption~\ref{ass.pos} hold. Then, a constant control input $u(t) = u^\star$ solves Problem~\ref{prob.opt}.
\end{theorem}

\begin{proof}
Let $ u^\star$ minimize $J(u)$ over the space of {\em constant\/} functions. Since~\eqref{eq.sys} with a constant control input is an LTI system, the maximum power-amplification coincides with the $\mc H_\infty$ norm. Moreover, because~\eqref{eq.sys} is a positive system, the maximal singular value of the frequency response matrix peaks at zero temporal frequency and the worst-case disturbance is a constant signal~\cite{tanaka2011bounded}. This implies that $J(u^\star) = \norm{z_{ u^\star, \bar d} }_{\textnormal{pow}}^2$ where $\bar d = v\geq 0$ is a constant nonnegative disturbance and $v$ is the right principal singular vector of the matrix $-Q^{1 \over 2}(A+D( u^\star))^{-1}B$.

Suppose there exists a time varying signal $\hat u\in \mc B_2$ such that $J(\hat u) < J( u^\star)$. Then, since $J$ measures the worst-case disturbance amplification, $\hat u$ must also decrease the power norm of the output of system~\eqref{eq.sys-z} for a constant disturbance $\bar d=v$, i.e.,
\be
	\norm{z_{\hat u, \bar d}}_{\textnormal{pow}}^2
	\;\leq\;
	J(\hat u)
	\;<\;
	\norm{z_{u^\star, \bar d}}_{\textnormal{pow}}^2
	\;=\;
	J(u^\star).
    \label{eq.bound}
\ee
We next show that this is not possible.

Every bounded $u \in \mc{B}_2$ can be written as $u = \bar u + \tilde u$ where $\bar u$ is constant and $\tilde u$ is bounded and zero-mean. Since $u^\star$ minimizes $\norm{z_{u, \bar d} }_{\textnormal{pow}}^2$ over constant control inputs and Lemma~\ref{lem} implies that $u^\star$ minimizes it over bounded zero-mean control inputs, $u^\star$ is a local minimizer of $\norm{z_{u, \bar d} }_{\textnormal{pow}}^2$ over all bounded $u \in \mc{B}_2$. By Lemma~\ref{lem.convex.l2}, this implies that $u^\star$ is a global minimizer of $\norm{z_{u, \bar d} }_{\textnormal{pow}}^2 $, contradicting~\eqref{eq.bound} and completing the proof.
\end{proof}

	\vspace*{-2ex}
\subsection{Subgradient algorithm}

We next cast Problem~\ref{prob.opt} as a finite-dimensional convex optimization problem. For constant control inputs we derive a subdifferential set of the objective function and employ a subgradient method to compute an optimal solution. Since subgradient algorithms have slow convergence rate, we provide conditions under which the objective function is differentiable. In this case, a variety of standard tools including gradient descent and quasi-Newton methods can be readily employed to ensure faster convergence.

{Theorem~\ref{thm.const.opt} implies that Problem~\ref{prob.opt} can be solved by minimizing $J(u)$ over constant control inputs. As mentioned above, for a constant control input,~\eqref{eq.sys} is} a positive system and the maximum power amplification is caused by a constant disturbance $\bar d = v$ where $v$ is the right principal singular vector of the matrix $-Q^{1\over2}(A+D(\bar u))^{-1}B$. For a constant control input $\bar u$, $\norm{ R^{1\over 2}u(t)}_{\textnormal{pow}}^2 = \bar u^TR \, \bar u$, and the objective function is given by
\[
	J(\bar u)
	~=~
	\bar{\sigma}^2
	\left(
	-Q^{1 \over 2} \left(A + D(\bar u)\right)^{-1}B
	\right)
	\;+\;
	\bar u^TR \, \bar u.
\]
Since $J(\bar u)$ is not always differentiable, subgradient methods~\cite{bertsekas1999nonlinear} can be used to find an optimal solution. Before we provide an explicit characterization for the subgradient set of $J$, we determine the adjoint of the linear operator $D$.
	
The adjoint of a linear operator $D$: $\R^m \to \D^n \subset \R^{n \times n}$ is the linear operator $D^{\dagger}$: $\R^{n \times n} \to \R^m$ which satisfies
\[
	\inner{X}{D(u)}
	\;=\;
	\inner{D^{\dagger}(X)}{u}
\]
for all $u \in \R^m$ and $X \in \R^{n \times n}$. Without loss of generality, the operator $D$ can be expressed as
\[
	D(u)
	\;=\;
	\ds\sum_{k \, = \, 1}^m D_k \, u_k
	\;=\;
	\diag \, (D_u u)
\]
{where $\diag \, (x)$ denotes a diagonal matrix with the vector $x$ on its main diagonal and the $k$th column of $D_u \in \R^{n \times m}$ is given by the main diagonal of the matrix $D_k \in \D^n$.} {The adjoint of the operator $D$ is then,}
\be
	D^\dagger \left(X\right)
	\;=\;
	D_u^T \, \diag \, (X)
	\label{eq.dadj}
\ee
where $\diag \, (X) \in \R^n$ is the vector determined by the diagonal elements of $X$.

\begin{proposition}
\label{prop:gradhinf}
Let $D$ be a linear operator and $A_{\mathrm{cl}} \DefinedAs A + D(\bar u)$ be Hurwitz. Then,
\be
	\label{eq.grad_hinf}
	\partial J(\bar u)
	\;=\;
	\bigg\{ 2\bar\sigma_{\mathrm{cl}}\ds\sum_i \alpha_i \, D^{\dagger}\left(A_{\mathrm{cl}}^{-1}Bv_i w_i^T  Q^{1/2}  A_{\mathrm{cl}}^{-1}\right)
	\;+\;
	2 R\bar u ~
	~\bigg|~ w_i^T \left(-Q^{1/2} A_{\mathrm{cl}}^{-1}B\right)v_i \,=\, \bar \sigma_{\mathrm{cl}}, ~ \alpha \,\in\, \mathcal{P} \bigg\}
\ee
where $\bar \sigma_{\mathrm{cl}} \DefinedAs \bar{\sigma} (-Q^{1/2} A_{\mathrm{cl}}^{-1}B)$
	and
\[
	\mc P
    \;\DefinedAs\;
    \bigg\{\alpha
	~|~
	\alpha_j \, \geq \, 0,
	~
	 \sum_j \alpha_j \, = \,1
    \bigg\}.
\]
\end{proposition}

    \begin{proof}
The maximum singular value $\bar{\sigma}(X)$ of a matrix can be expressed as~\cite{boyd2004convex},
\[
	\bar{\sigma}(X)
	\;=\;
	\sup_{\norm{w} = 1, \norm{v} = 1} w^T Xv.
\]
The subdifferential set of the supremum over a set of differentiable functions,
\[
    f(x)
    \;=\;
    \sup_{i \, \in \, \mc I}
    ~
    f_i(x)
\]
is the convex hull of the subgradients of each function $f_j$ that achieves the maximum~\cite[Theorem 1.13]{shor2012minimization},
\[
    \partial f(x)
    \;=\;
    \ds\sum_{j | f_j(x) = f(x)} \alpha_j\nabla f_j(x),
    ~~
    \alpha \, \in \, \mc P.
\]
Therefore, the subgradient of $\bar{\sigma}$ is given by
\[
	\partial \bar{\sigma}(X)
	\;=\;
	\left\{ \sum_j \alpha_j w_jv_j^T  ~|~ w_j^T Xv_j = \bar{\sigma}(X), ~\alpha \in \mathcal{P} \right\}.
\]
The matricial derivative of $X^{-1}$ and the application of the chain rule yield~\eqref{eq.grad_hinf}.
\end{proof}

\begin{remark}
Proposition~\ref{prop:gradhinf} holds for any linear operator $D$: $\R^m
\to \R^{n \times n}$. Since we are primarily concerned with operators $D$: $\bbR^m \to \D^n$, $D^\dagger$ in~\eqref{eq.grad_hinf} can be replaced with~\eqref{eq.dadj}.
\end{remark}

In general, $J$ is nondifferentiable. To find a solution to Problem~\ref{prob.opt}, we implement a subgradient algorithm~\cite{bertsekas1999nonlinear},
\[
	u^{k+1}
	\;=\;
	u^k
	\;-\;
	\alpha_k \,
	d^k
\]
where $\alpha_k > 0$ is a step size and $d^k \in \partial J(u^k)$ is a member of the subdifferential set of the objective function. For a small enough $\alpha$, this algorithm is guaranteed to converge with convergence rate $\mc{O}(1/\sqrt{k})$. Backtracking is used to ensure closed-loop stability, but, since the subgradient algorithm is not a descent method, sufficient descent criteria such as the Armijo rule may not be employed.

\begin{remark}
When the graph associated with $A$ is strongly connected, $J$ is continuously differentiable~\cite[Proposition 9]{dhingra2016convexity} and therefore a variety of methods, including gradient descent, Newton's method, and quasi-Newton methods~\cite{bertsekas1999nonlinear} can be used to find the solution more efficiently.
\label{rem.con}
\end{remark}

The computational complexity of each iteration in the subgradient algorithm is $\mc{O}(n^3)$ because of matrix inversion and the singular value decomposition. This computational cost is of the the same order as that of evaluating $J$. {Finally, we refer the interested reader to~\cite[Section VI]{dhingra2016convexity}, where we discuss the use of proximal algorithms~\cite{becteb09,parboy13} to solve optimization problems of the form of Problem~\ref{prob.opt} augmented with structure-promoting nonsmooth regularizers.}

	\vspace*{-2ex}
\section{Robust optimal control problem}
	\label{sec.robust}

In this section we study the robust optimal control problem for system~\eqref{eq.sys} with model uncertainty. We consider uncertain dynamics of the form
\be\label{eq.hiv.delta}
\dot{x} ~=~ \left((A+\Delta_A)\;+\;  D(u - \delta_{u})\right)x ~+~ Bd,
\ee
where
$$
\Delta_A = \thbth
{\delta_{11}}{\cdots}{\delta_{1n}}
{\vdots}{\ddots}{\vdots}
{\delta_{n1}}{\cdots}{\delta_{nn}}
$$
represents the uncertainty in the matrix $A$ and the vector $\delta_{u}$ represents the uncertainty in the input $u$. {The input uncertainty $\delta_u$ is relevant in biological applications as explained in Section~\ref{sec.hiv}. We assume that the uncertainties are bounded as $|\delta_{ij}| \le \alpha_{ij}$ for all $(i,j)$} and $|\delta_{u_k}|\le\beta_k$ with $\alpha_{ij}\geq0$ and $\beta_k\geq0$. We define the set of admissible perturbations as
$$
\bold \Delta \DefinedAs \left\{ (\Delta_A,\delta_{u}) \left|\; {|\delta_{ij}|}\le\alpha_{ij}, ~ |\delta_{u_k}|\le\beta_k \right.\right\}
$$
and
$$
	\tilde A
	\;\DefinedAs\;
	 \thbth
	{\alpha_{11}}{\cdots}{\alpha_{1n}}
	{\vdots}{\ddots}{\vdots}
	{\alpha_{n1}}{\cdots}{\alpha_{nn}},
	\quad
	\beta
	\;\DefinedAs\;
	 \thbo{\beta_1}{\vdots}{\beta_m}.
$$
For {a fixed $\Delta_A$, $\delta_{u}$ and $u$,} we denote by $J (u\,;\,\Delta_A,\delta_{u})$ the power-induced gain of system~\eqref{eq.hiv.delta}. We are now ready to state the main result of this section.

\begin{theorem}\label{thm.rob}
Let Assumption~\ref{ass.pos} hold and let $D$: $\R^m \to \D^n$ be such that $-D(u) \in \D_+$ for all $u \in \R_+$.
Provided there exists a $\bar u$ such that $(A + \tilde A - D(\beta)) + D(\bar u) $ is Hurwitz, the solution to the robust optimal control problem
$$
\minimize_{u}~\maximize_{(\Delta_A,\delta_{u})\in\bold \Delta}~ J (u\,;\,\Delta_A,\delta_{u})
$$
is given by the solution to Problem~\ref{prob.opt} applied to the system
$$
\dot{x} ~=~ {\left((A \;+\; \tilde A \;-\; D(\beta)) \;+\; D(\bar u) \right)} x ~+~ Bd.
$$
\end{theorem}

We first state an established result which will be useful in proving Theorem~\ref{thm.rob}.

\begin{lemma}[{\cite[Theorem 8.1]{horn1985matrix}}]\label{lem.power}
The following statements are true:

\begin{enumerate}[a)]
\item Let $A\in\R^{n\times n}$ then $| \mathrm e^{A} |\leq  \mathrm e^{|A|}=|\mathrm e^{|A|}|$.
\item Let $A,B\in\R^{n\times n}$, if $0\leq A\leq B$ then $\mathrm e^{A} \leq \mathrm e^{B}$.
\item Let $A,B\in\R^{n\times n}$ if $|A|\leq |B|$ then $\bar\sigma(A)\leq \bar\sigma(B)$.
\end{enumerate}
\end{lemma}

\begin{proof}[Proof of Theorem~\ref{thm.rob}]
First we recall that we can write $D(u) = \diag(D_u u)$. Given a constant $u(t)=\bar u$, define
$$
    \mc A_\Delta
    \;\DefinedAs\;
    (A \,+\, \Delta_A)
    \;+\;
    D(\bar u \,-\, \delta_{u})
$$
and
$$
    \mc A
    \;\DefinedAs\;
    (A \,+\, \tilde A)
    \;+\;
    D_u(\bar u \,-\, \beta).
$$
Note that $|\mc A_\Delta| \leq \mc A$. We begin by showing that $\mc A_\Delta$ is Hurwitz for all $(\Delta_A,\delta_{u})\in\bold \Delta$ if and only if $\mc A$ is Hurwitz. Necessity is obvious as $(\tilde A,\beta)\in\bold \Delta$. To prove sufficiency, assume $\mc A$ is Hurwitz. Since $\mc A$ is Metzler, we know that there exists a vector $p>0$ such that
$$
\mc A \, p \, < \, 0,
$$
which is equivalent to
$$
    \left|a_{ii}+\alpha_{ii} + \bar u_{i} - \beta_i\right|p_i
    \;>\;
    \sum_{j\neq i} \left|a_{ij}+\alpha_{ij}\right|p_j .
$$
Clearly for all $(\Delta_A,\delta_u)\in\bold \Delta$,
$$
    \left|a_{ii}+\delta_{ii} + \bar u_{i} - \delta_{u_i}\right|p_i
    \;>\;
    \sum_{j\neq i} \left|a_{ij}+\delta_{ij}\right|p_j,
$$
which means that $\mc A_\Delta$ is \emph{scaled diagonally dominant} and therefore Hurwitz.

\vsp

Now that we know that  $\mc A_\Delta$ is stable for all $(\Delta_A,\delta_{u})\in\bold \Delta$, we show that $(\tilde A,\beta)$ is also the worst case perturbation for the performance index $J$. Let $\alpha>0$ be such that $\mc A+\alpha I\geq 0$. Note that $| \mc A_\Delta+\alpha I | \leq \mc A+\alpha I$. Let us consider  $J (u\,;\,\Delta_A,\delta_{u})$ for a constant $u(t)=\bar u$ such that $\mc A$ is Hurwitz. Since $\mc A_\Delta$ is also Hurwitz we conclude that

\begin{subequations} \label{eq.worst_case}
\begin{alignat}{1}
  J ( u\,;\,\Delta_A,\delta_{u})
  & \;=\;
  \bar\sigma^2\left(Q^{1\over 2}\int_0^\infty\mathrm e^{\mc A_\Delta \, t} \mathrm dt\, B\right)
  \;+\;
  \bar u^TR\bar u
  \\
  & \;=\;
  \bar\sigma^2\left(Q^{1\over 2}\int_0^\infty \mathrm e^{-\alpha t}\mathrm e^{(\mc A_\Delta +\alpha I)\, t}\mathrm dt\, B\right)
  \;+\;
  \bar u^TR\bar u\\
  & \;\leq\;
  \bar\sigma^2\left(Q^{1\over 2}\int_0^\infty\mathrm e^{-\alpha t}\mathrm e^{|\mc A_\Delta +\alpha I | \, t}\mathrm dt\, B\right)
  \;+\;
  \bar u^TR\bar u \label{eq.absval}\\
  & \;\leq\;
  \bar\sigma^2\left(Q^{1\over 2}\int_0^\infty \mathrm e^{-\alpha t}\mathrm e^{(\mc  A  \, +\alpha I) \,t}\mathrm dt\, B\right)
  \;+\;
  \bar u^TR\bar u \label{eq.tilde}\\
  & \;=\;
  J ( u\,;\,\tilde  A, \beta),
\end{alignat}
\end{subequations}
where~\eqref{eq.absval} follows from the {application of statements (a) and (b) in Lemma~\ref{lem.power}, while~\eqref{eq.tilde} follows from the application of (b) and (c) of Lemma~\ref{lem.power}.}

From~\eqref{eq.worst_case} we know that, given any constant input $u(t) = \bar u$, the worst case perturbation $(\Delta_A,\delta_{u})\in\bold \Delta$ is given by $(\tilde A,\beta)$. By Theorem~\ref{thm.const.opt} the optimal $u(t)$ in response to the perturbation $(\tilde A,\beta)$ is constant and therefore the proof is complete.
\end{proof}

The statement of Theorem~\ref{thm.rob} is intuitive when considering a positive system. The worst case uncertainty is given when every component of $\Delta_A$ hits its upper bound and every component of $\delta_u$ hits its lower bound. Similar results without such intuitive solutions can be derived for more complex uncertainty descriptions where not all elements of $\Delta_A$ and $\delta_u$ are allowed to vary independently. For a deeper discussion on this topic we refer the reader to~\cite{ColSmi:2016:IFA_5242} or~\cite{colombino2015robust} where robust stability for positive systems is considered in a more general setting with coupled linear and nonlinear uncertainties, respectively.

	\vspace*{-2ex}
\section{Combination drug therapy for HIV}
\label{sec.hiv}

As shown in~\cite{hernandez2011discrete, jonmatmur14}, the mutation-replication dynamics of HIV can be modeled as bilinear system~\eqref{eq.sys-z}. Here, the $i$th component of the state vector $x$ represents the population of the $i$th HIV mutant. The matrix $A$ specifies the mutation and replication rates of the virus; $A_{ij}$ represents the rate at which mutant $i$ turns into mutant $j$ and $A_{ii}$ represents the net replication rate of mutant $i$. The control input $u_k$ is the dose of drug $k$ and $D(\mre_k)$ specify how efficiently drug $k$ neutralizes each HIV mutant. The exogenous input $d$ represents the effect of noise, disturbances or unmodeled dynamics on the virus population. The performance output~\eqref{eq.z} captures the lethality each virus strand via $Q$ and the penalizes the magnitude of the drug dosages via $R$.

	\vspace*{-2ex}
\subsection{Uncertain model of HIV dynamics}

A large challenge with HIV virus dynamics is model uncertainty. Uncertainty arises from two different aspects:
\begin{itemize}
\item It is difficult to accurately estimate replication and mutation rates. Moreover, these rates vary between patients.
\item It is difficult to precisely deliver drugs. Maintaining the target drug concentration at the treatment site is itself a challenging control problem.
\end{itemize}
In dynamics~\eqref{eq.hiv.delta}, $\Delta_A$ represents the uncertainty in the replication-mutation rates  and $\delta_{u}$ represents uncertainty in the delivered drug dosage.

Theorem~\ref{thm.rob} has a clear physical interpretation for this application; the worst case perturbation is the one for which the virus replicates and mutates most aggressively and the drugs are the least effective. Furthermore, Theorem~\ref{thm.const.opt} implies that the optimal treatment strategy for minimizing the effect of disturbances is maintaining a constant concentration of drugs.

	\vspace*{-2ex}
\subsection{An example}

In this section we present a simple example to illustrate the utility of the robust optimal control formulation. Consider the system described in Section~\ref{sec.hiv} for a virus with $n$ mutagens whose mutation pattern is given by the path graph in Figure~\ref{fig.rob} where $r$ represents the replication rate of each mutagen. The entry $c \geq 0$ represents a mutation pathway which is nominally zero but whose value is uncertain, i.e. $\Delta_A = c\,\mre_1\mre_n^T$. The control authority $u$ represents a drug which treats all the mutagens identically with no uncertainty, i.e. $D(u) = -uI$ and $\beta = 0$. The input disturbance matrix is given by $B = \mre_n$, the state penalty matrix is $Q = \mre_n \mre_n^T$, and the control input penalty matrix is $R = \rho I$.

Uncertainty in the model is natural when the mutation pathways are not completely understood and can be a convenient way to account for nonlinear dynamics.
The uncertain mutation pathway from mutagen $x_n$ to $x_1$ is particularly important because a nonzero value introduces a cycle into the graph associated with $A$. The optimal drug dose designed for the nominal model may not be stabilizing when applied to the true dynamics.

Since $A + D(u)$ is lower triangular when $c=0$, its eigenvalues are the diagonal entries and therefore the closed-loop system is nominally stable when $r - u < 0$.
When $c \not= 0$, the graph associated with $A + D(u)$ is a strongly connected graph and therefore has a unique largest eigenvalue~\cite[Proposition 9]{dhingra2016convexity}. It follows that any value of $c$ which changes the sign of the determinant,
\[
    \det (A \,+\, D(u))
    \;=\;
    (r \,-\, u)^n \;+\; (-1)^{n-1}c
\]
relative to $\det(A + D(u))$ when $c=0$, i.e., any $c$ such that
\[
c ~\geq~ |(r \,-\, u)^n|
\]
causes the system to be unstable.

By computing the $n1$th entry of $(A + D(u))^{-1}$, the objective function in Problem~\ref{prob.opt} can be explicitly expressed as,
\[
    J(u)
    \;=\;
    \left(
    \dfrac{|(u\,-\,r)^{n-1}|}{|(u\,-\,r)^n| \;-\; c}
    \right)^2
    \;+\;
    \rho
    u^2.
\]
In the nominal case when $c = 0$, $J(u) = (r + u)^{-2} + \rho u^2$, so the optimal value of $u$ solves the quartic equation,
\[
    (r \,+\, u)^{-3} \,-\, \rho u \;=\; 0
\]

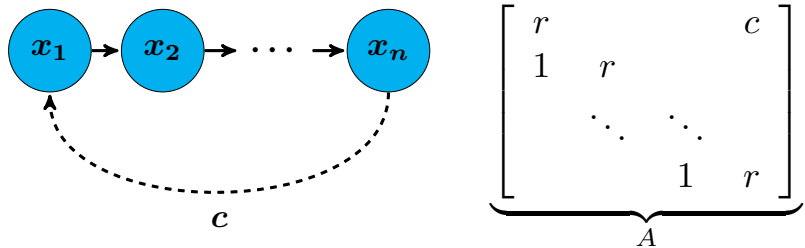
\begin{figure}
\centering
\[
	\begin{aligned}
	\raisebox{-17mm}{
	\resizebox{.35\columnwidth}{!}{
	\begin{tikzpicture}[>=stealth',shorten >=1pt, node distance=1.8cm, on grid, initial/.style={},font=\boldmath]

  \node[state,scale=1.5,fill=ProcessBlue]          (1)                        {$x_1$};
  \node[state,scale=1.5,fill=ProcessBlue]          (2) [right =of 1]    {$x_2$};
  \node[scale=1.5]                                  (dots) [right =of 2]    {$\cdots$};
  \node[state,scale=1.5,fill=ProcessBlue]          (3) [right =of dots]    {\textbf{$x_n$}};

\tikzset{every node/.style={fill=white}}

\tikzset{mystyle/.style={->,double=black}}
\path
        (1)     edge [mystyle]    (2)
        (2)     edge [mystyle]    (dots)
        (dots)  edge [mystyle]    (3);

\tikzset{mystyle/.style={dashed,->,relative=false,in=-90,out=-90,double=black}}
\path (3)     edge [mystyle] node[below=0.15cm,scale=1.5,] {$c$}   (1);

\end{tikzpicture} }}
	&
	~~~
	&
	\resizebox{.25\columnwidth}{!}{$
	\underbrace{
	    \matbegin \begin{array}{cccc}
                r &  &  & c \\
                1 & r &  &   \\
                  & \ddots & \ddots & \\
                  & & 1 & r
                \end{array}\matend
                }_{A}$}
        \end{aligned}
        \]
\caption{The dynamics are described by a path graph. The dashed line in the diagram and the $c$ entry in the matrix represent an uncertain link from $x_n$ to $x_1$ which would introduce a cycle into the mutation pattern.}
\label{fig.rob}
\end{figure}

We consider the system in Fig.~\ref{fig.rob} with $n = 10$ mutants, replication rate $r = 1$, and control penalty $\rho = 3$. A nominally optimal drug dose $u_{\text{nom}}$ was designed to minimize $J(u)$ for the nominal model (i.e., $c = 0$) and a robust optimal drug dose $u_{\text{rob}}$ was designed to solve the robust optimal control problem considered in Theorem~\ref{thm.rob} with the assumption $c \leq 0.1$ (i.e., $\alpha_{1n} = 0.1$ and all other $\alpha_{ij} = 0$, or, equivalently, $\tilde A = 0.1\mre_1\mre_n^T$).
The designed doses were,
\[
u_{\text{nom}}
~=~
1.5936,
\hspace{1cm}
u_{\text{rob}}
~=~
1.9413.
\]
Any $c \geq 0.0054$ destabilizes a system controlled by $u_{\text{nom}}$ and any $c \geq 0.5461$ destabilizes a system controlled by $u_{\text{rob}}$.

Figure~\ref{fig.robnom} shows the impulse response for the nominal model controlled by $u_{\text{nom}}$ and $u_{\text{rob}}$. Figure~\ref{fig.robrob} shows the closed loop impulse response when $c = 0.02$. The nominal controller is more sensitive to perturbations in the model.

\begin{figure}[htb]
    \begin{subfigure}[b]{1\columnwidth}
        \centering
        \includegraphics[width=0.5\columnwidth]{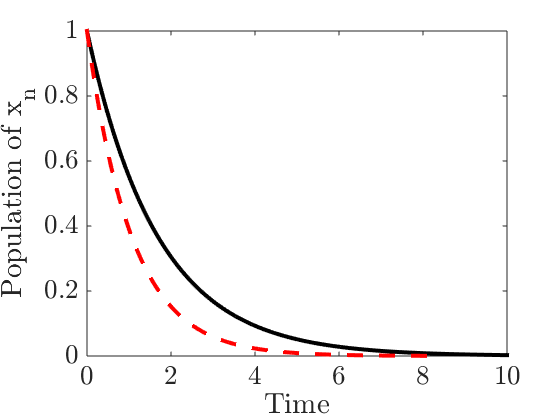}
        \caption{Closed loop impulse responses of the nominal system}
        \label{fig.robnom}
    \end{subfigure}%
    \\[0.5cm]
    \begin{subfigure}[b]{1\columnwidth}
        \centering
        \includegraphics[width=0.5\columnwidth]{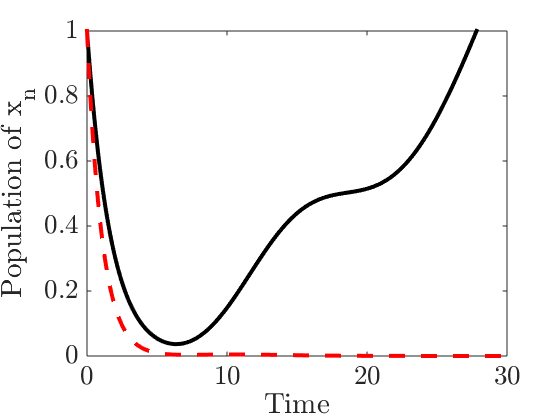}
        \caption{Closed loop impulse responses of the system with $c = 0.02$}
        \label{fig.robrob}
    \end{subfigure}
    \caption{Closed loop impulse response in the (a) nominal and (b) perturbed model with the nominal controller $u_{\text{nom}}$ ({\bf black ---}) and the robust controller $u_{\text{rob}}$ (\textcolor{red}{\bf red - - -}).}
    \label{fig.gamvar}
\end{figure}

	\vspace*{-2ex}
\section{Concluding remarks}
	\label{sec.conc}

We study an infinite horizon optimal control problem for a class of monotone bilinear systems. The objective function is given by the induced power norm which represents a generalization of the $\cH_\infty$ performance metric. We prove that the optimal control signal is constant over time, show that the optimal controller can be computed by solving a finite-dimensional non-smooth convex optimization problem, and develop a subgradient algorithm to compute the optimal solution. We further extend our results for systems with uncertain dynamics and show that computing the optimal robust controller is no harder than computing the nominal controller. Finally, we apply our results to an example inspired by combination drug therapy design for the treatment of HIV.


	\vspace*{-2ex}
\section*{Acknowledgments}

The authors would like to acknowledge Peyman Mohajerin Esfahani and Kevin Nowland for helpful discussions on functional analysis.

\vspace*{-2ex}
\setstretch{1.1}

\end{document}